\documentclass[11pt]{article}
\usepackage{amsmath,amssymb}
\usepackage{amsthm}
\usepackage{inputenc}
\usepackage{amsthm,amsfonts,amsmath,verbatim,amssymb,verbatim,cite}

\usepackage{mathrsfs}
\usepackage[usenames]{color}
\usepackage{tikz,graphicx}

\numberwithin{equation}{section}
\newtheorem{theorem}{Theorem}[section]
\newtheorem{proposition}[theorem]{Proposition}
\newtheorem{lemma}[theorem]{Lemma}
\newtheorem{definition}[theorem]{Definition}
\newtheorem{claim}[theorem]{Claim}
\newtheorem{corollary}[theorem]{Corollary}
\newtheorem{conjecture}[theorem]{Conjecture}

\newtheorem{problem}[theorem]{Problem}
\newtheorem{algo}[theorem]{Algorithm}
\newtheorem{nota}[theorem]{Notation}
\newtheorem{construction}[theorem]{Construction}
\newtheorem{remark}[theorem]{Remark}
\newtheorem{obs}[theorem]{Observation}

\def\mi{\hbox{\rm mi}}

\begin{document}

%\begin{frontmatter}



\title{On the number of $k$-dominating independent sets}


\author{Zolt\'an L\'or\'ant Nagy \thanks{MTA--ELTE Geometric and Algebraic Combinatorics Research Group, H--1117 Budapest, P\'azm\'any P.\ s\'et\'any 1/C, Hungary. Email: nagyzoli@cs.elte.hu. }}

\date{}
%\author{Zolt\'an L\'or\'ant Nagy}
%\address{ Budapest,  Hungary}
%\ead{nagyzoltanlorant@gmail.com}

\maketitle

\begin{abstract} We study the existence and the number of $k$-dominating independent sets in certain graph families. While the case $k=1$  namely the case of maximal independent sets  - which is originated from Erd\H{o}s and Moser - is widely investigated, much less is known in general. In this paper we settle the question for trees and prove that the maximum number of $k$-dominating independent sets in $n$-vertex graphs is between $c_k\cdot\sqrt[2k]{2}^{\,n}$ and $c_k'\cdot\sqrt[k+1]{2}^{\,n}$ if $k\geq 2$, moreover the maximum number of $2$-dominating independent sets in $n$-vertex graphs is between $c\cdot 1.22^{\,n}$ and $c'\cdot1.246^{\,n}$.
Graph constructions containing a  large number of $k$-dominating independent sets are coming from product graphs, complete bipartite graphs and finite geometries. The product graph construction is associated with the number of certain MDS codes.
\end{abstract}


%\end{frontmatter}

% \linenumbers


\bigskip
\section{Introduction and background}
\label{1}

Let $G= G(V, E)$ be a simple graph. For any vertex $v\in V(G)$, $d(v)$ denotes the degree of $v$, $N(v)$ denotes the set of neighbors of $v$, and $N[v]$ denotes the closed neighborhood, i.e. $N[v]:=N(v)\cup \{v\}$.

A subset $I\subseteq V(G)$ is called \textit{independent} if it does not induce any edge. A \textit{maximal independent set} is an independent set which is not a proper subset of another independent set (it cannot be extended). A maximum independent set is an independent set of maximal size; its size is denoted by $\alpha(G)$.

A subset $D\subseteq V(G)$ is a \textit{dominating} set in $G$ if each vertex in $V(G)\setminus D$ is adjacent to at least one vertex of $D$, that is, $\forall v\in  V(G)\setminus D$, $|N(v)\cap D|\geq 1$.
 We call a set \textit{$k$-dominating} if each vertex in $V(G)\setminus D$ is adjacent to at least $k$ vertices of $D$, that is, $\forall v\in  V(G)\setminus D$, $|N(v)\cap D|\geq k$ .
 The theory of independent sets and dominating sets has been studied extensively over the last 60 years.
 
Following the concept of W\l{}och \cite{wloch}, we study \textit{$k$-dominating independent sets}, or $k$-DISes for brevity, in case $k>1$. Note that the case $k=1$ when a set $W$ is dominating and independent at the same time is also extensively studied. These sets are called \textit{kernels} of the graphs (due to Neumann and Morgenstern) and they clearly  coincide with the maximal independent sets.
The possible number of kernels has been resolved in many graph families  including connected graphs, bipartite graphs and trees, triangle-free graphs, see the results of Moon, Moser, F\"uredi, Hujter and Tuza, Jou and Chang \cite{chang1, chang2, furedi, HT,  moon}.
\medskip

Our principal function  is formulated in  the following

\begin{nota}
Let $\mi_k(n)$ denote the maximum number of $k$-DISes in graphs of order $n$, and let
$\mi_k(n, \mathcal{F})$ denote the maximum number of $k$-DISes in the $n$-vertex members of the graph family  $\mathcal{F}$. If $\mathcal{F}$ consists of a single graph $G$, we denote by $\mi_k(G)$ the  number of $k$-DISes in $G$.
\end{nota}

Concerning graph constructions, we will use 
\begin{nota}
For arbitrary graphs $G$ and $H$, $G+H$ denotes the disjoint union of $G$ and $H$. Similarly, if a parameter $k\in \mathbb{Z}^+$ is given, $kG$ denotes the disjoint union of $k$ copies of $G$.  $K_m\square K_m$ denotes the Cartesian product of two $K_m$ graphs, or in other words  it is the strongly regular Lattice graph $L(m)$, or Rook graph. Finally, $(K_m)^t$ denotes the Cartesian product of $t$ $ K_m$ graphs: $K_m\square K_m\square \ldots \square K_m$.
\end{nota}

\begin{obs}\label{obsit} $\mi_k(G+H)=\mi_k(G)\cdot\mi_k(H)$ for any two graphs $G$ and $H$. 
\end{obs}

\begin{nota}
Let $\zeta_k(G):=\sqrt[n]{\mi_k(G)}$ for a fixed graph $G$ on $n$ vertices and let $$\zeta_k(n):=\sqrt[n]{\mi_k(n)}, \ \  \zeta_k(n,  \mathcal{F}):=\sqrt[n]{\mi_k(n, \mathcal{F})}.$$
\end{nota}

\begin{theorem}\label{alap0}
\begin{itemize}
\item[(i)]  $\zeta_k(n) \in [1,2]  \ \ \forall k,n \in \mathbb{Z}^+, k\leq n.$

\item[(ii)]  $\zeta_k(G) \leq \lim \inf \zeta_k(n) \ \  \forall k  \in \mathbb{Z}^+$ and for every fixed graph $G$.

\item[(iii)] $\forall k  \ \ \exists \lim \zeta_k(n)$.

\end{itemize}
\end{theorem}

\begin{proof} Part $(i)$ is straightforward since $1\leq\mi_k(n)\leq 2^n$ in view of the empty graph and the number of all possible subsets of the vertex set.\\
 Suppose $\zeta_k(G)\geq 1$. If we apply Observation \ref{obsit} to $\left\lfloor \frac{n}{|V(G)|}\right\rfloor$ disjoint copies of $G$ and suitable number of additional isolated vertices, we get $\mi_k(n)\geq \mi_k(G)^{\left\lfloor\frac{n}{|V(G)|}\right\rfloor}$, hence part $(ii)$ follows.\\
Finally, part $(i)$ and part $(ii)$ together implies part $(iii)$.
\end{proof}

Our main theorems are 

\begin{theorem} \label{fo1} The order of magnitude of the maximum number of $2$-DISes is bounded as follows.
 $$1.22<\sqrt[9]{6}\leq \lim \zeta_2(n) \leq  \sqrt[5]{3}<1.2457.$$
\end{theorem}

\begin{theorem} \label{fo2} For every $k>2$,
 $$\sqrt[2k]{2}\leq \lim\zeta_k(n) \leq  \sqrt[k+1]{2}.$$
\end{theorem}

The paper is built up as follows.
Section 2 summarizes the main known results on the number of $k$-DISes  for $k=1$.

In Section 3, we give a simple characterization of graphs which contain a $k$-dominating independent set and point out the existence of large graph families not containing $2$-DISes. Next we prove that if a $k$-DIS exists in a tree, then it is unique. Furthermore we present an efficient algorithm which provides a $k$-DIS in a given tree or proves the non-existence of such a set. Finally, we present graph constructions containing many $k$-DISes. Proposition \ref{expect} essentially states that a random graph contains a huge number of $k$-DISes for any fixed $k$. 

We prove the lower bounds of Theorem \ref{fo1} and Theorem \ref{fo2} in Section 4.
These bounds are based on constructions. The presented graphs providing the lower bound on $ \zeta_k(n)$ are of different structure in the cases $k=1, 2, 3$ and $k\geq 4$.
 One of them leads to the determination of the number of ternary $(n, M, 2)_3$ MDS codes as well. 

Extremal constructions are often obtained from finite geometry. (For  detailed descriptions we refer to \cite{FS}.) In our case,  specific examples 
for different types of graphs with many $2$-dominating ($k$-dominating) independent sets are given based on hyperovals and generalized $\{k;n\}$-arcs.

 Section 5 is devoted to the upper bound part of Theorem \ref{fo1} and Theorem \ref{fo2}. 
 At last, some open questions and concluding remarks are collected in Section 6.

%\begin{example}
%Let $G$ be a hypercube of dimension $d$, that is, the Carthesian product of $t$ edges; $G=K_2\square K_2 \square %\ldots \square K_2$. Then  $G$ has (exactly) $2$ $k$-dominating independents sets for every $1<k\leq t$.
%\end{example}

%-----------------------------------------------------------------------------------

\bigskip
\section{Results on the number of maximal independent sets: $1$-DISes}
\label{2}

Erd\H os and Moser raised the question to determine the maximum number of maximal cliques in $n$-vertex graphs.  Note that it is the same as the maximum number of maximal independent sets (that is, $1$-DISes) an $n$-vertex graph can have.

Answering a question of Erd\H os and Moser,  Moon and Moser proved the following well known

\begin{theorem}[Moon-Moser, \cite{moon}]\label{alap} The following equality holds:

$$\mi_1(n)= \left\{ \begin{array}{lll} 3^{n/3} & \textrm{if } n\equiv 0 \pmod 3 \\ \frac{4}{3}\cdot3^{ \lfloor n/3 \rfloor} & \textrm{if } n\equiv 1 \pmod 3\\ 
2\cdot3^{ \lfloor n/3 \rfloor} & \textrm{if } n\equiv 2 \pmod 3 \end{array} \right.$$
\end{theorem}
\noindent Moreover, they proved that the equality is attained if and only if the graph $G$ is isomorphic to the graph $\frac{n}{3} K_3$  (if $n\equiv 0 \pmod 3$); to one of the graphs $(\lfloor \frac{n}{3} \rfloor-1)K_3 + K_4$ or $(\lfloor \frac{n}{3} \rfloor-1)K_3 + 2K_2$ (if $n\equiv 1 \pmod 3$); to $\lfloor n/3 \rfloor K_3 + K_2$ (if $n\equiv 2 \pmod 3$).

\begin{corollary}  $\lim \zeta_1(n)=  \sqrt[3]{3}$,  and $\lim \zeta_k(n) \in [1, \sqrt[3]{3}]$ for all $k>1$.
\end{corollary}

%\begin{proof}

%\end{proof}

\noindent For connected graphs the question was raised by Wilf \cite{wilf},  and the answer is fairly similar.

\begin{theorem}[F\"uredi \cite{furedi}, Griggs,  Grinstead, Guichard \cite{griggs}] Let $\mathcal{F}_{con}$ be the family of connected graphs.
Then 

$$\mi_1(n, \mathcal{F}_{con})= \left\{ \begin{array}{lll} \frac{2}{3}\cdot3^{n/3}+\frac{1}{2}\cdot2^{n/3} & \textrm{if } n\equiv 0 \pmod 3 \\
3^{ \lfloor n/3 \rfloor}+ \frac{1}{2}\cdot2^{ \lfloor n/3 \rfloor}  & \textrm{if } n\equiv 1 \pmod 3\\ 
\frac{4}{3}\cdot3^{ \lfloor n/3 \rfloor}+\frac{3}{4}\cdot2^{ \lfloor n/3 \rfloor} & \textrm{if } n\equiv 2 \pmod 3 \end{array} \right.$$
\end{theorem}
The extremal graphs are determined as well. In these graphs, there is a vertex of maximum degree, and its removal yields a member of the extremal graphs list of Theorem \ref{alap}.
\medskip

 \noindent Wilf, and later Sagan studied the family of trees.

\begin{theorem}[\cite{wilf, sagan}]\label{tree} Let $\mathcal{T}$ be the family of  trees.
Then  the following equality holds:

$$\mi_1(n, \mathcal{T})= \left\{ \begin{array}{ll} \frac{1}{2}2^{n/2}+1 & \textrm{if } n\equiv 0 \pmod 2 \\ 2^{ \lfloor n/2 \rfloor} & \textrm{if } n\equiv 1 \pmod 2 \end{array} \right.$$

The extremal trees can be classified. 
\end{theorem}

\begin{corollary}  $\lim \zeta_1(n, \mathcal{T})=  \sqrt{2}$.
\end{corollary}

\begin{theorem}[Hujter, Tuza \cite{HT}]
Every triangle-free graph on $n \geq 4$ vertices has at most $2^{n/2} $ or $5 \cdot 2^{( n - 5 )/2} $ maximal independent sets, whether $n$ is even or odd. In each case, the extremal graph is unique.
\end{theorem}

\medskip

\section{$k$-DISes --- existence and characterizations }

While kernels ($1$-DISes)  obviously exist in every graph, this is far from being true for $k$-DISes for a fixed $k>1$.  To illustrate this phenomenon, consider

\begin{proposition} Let $G$ be (i) a  complete graph, (ii) an odd cycle, (iii) the complement of a connected triangle-free  graph with at least $2$ edges, w.r.t. $K_n$. Then $G$ does not contain a $k$-dominating independent set for $k>1$.   
\end{proposition}

\begin{proof}
It is straightforward to check the statement for (i) and (ii). If $G$ is the complement of a connected  triangle-free  graph, then an independent set consists of at most $2$ vertices in $G$, however no  pair of vertices are both adjacent to every other vertex in the graph.
\end{proof}

Hence the question naturally arises whether to contain (many) $k$-dominating sets is rather a rare property for $k>1$. 
Consider the Erd\H os-R\'enyi random graph $G_{n,p}$. Let $X_{t,1}$ denote the random variable which counts the number of maximal independent sets of size $t$ in $G_{n,p}$ and $X_{t,k}$ denote random variable which counts the number of $k$-DISes of size $t$ in $G_{n,p}$. \\
Following the idea of Bollob\'as and Erd\H os on maximal cliques \cite{EB}, one can easily calculate the expected value of $X_{t,1}$, $X_{t,2}$ or generally of $X_{t,k}$ as well. Note that the  expected value for $X_{t,1}$ is well known, we only add here for the purpose of comparison. 

\begin{proposition}\label{expect} $\mathbb{E}(X_{t,1})=\binom{n}{t}(1-p)^{\binom{t}{2}}\left(1-(1-p)^t\right)^{n-t},$ 

 \hspace{0.6cm} $\mathbb{E}(X_{t,2})=\binom{n}{t}(1-p)^{\binom{t}{2}}\left(1-(1-p)^t-tp(1-p)^{t-1}\right)^{n-t}$,

\hspace{0.6cm} $\mathbb{E}(X_{t,k})=\binom{n}{t}(1-p)^{\binom{t}{2}}\left(1-(1-p)^t-tp(1-p)^{t-1}- \cdots - \binom{t}{k-1}p^{k-1}(1-p)^{t-k+1}  \right)^{n-t}$.
\end{proposition}

\begin{corollary} \label{kov}
Let $p=1/2$. 

If $t<\log_2{n}-2\log\log n$ or $t >2\log_2{n}$, then $\mathbb{E}(X_{t,1})<1/n$  and  so $\mathbb{E}(X_{t,k})<1/n$ for all $k$.

If $t=c\log_{2}n$ with constant $1<c<2$, then 
 $\mathbb{E}(X_{t,k})= n^{\Omega( c(1-\frac{c}{2})\log_{2}n)}$ for all $k$.
\end{corollary}

Next, we present an easy constructive method to gain graphs with $k$-DISes.

\begin{construction}\label{konst}
Let $\Sigma=\{S_{t_i}\}$ be a set of disjoint stars with centers $v_i$ such that $t_i\geq k$.
Then the following two operations are allowed:
\begin{itemize}
\item[i] Identification of $x\in N(v_i)$ and $y\in N(v_j)$ in the $i$th and $j$th star for a pair $(i,j)$, 
\item[ii] Addition of an edge between $v_i$ and $v_j$ for a pair $(i,j)$.

\end{itemize}
\end{construction}

\begin{claim} 
Every graph $G$ which contains a $k$-dominating independent set can be  obtained by 
 Construction \ref{konst}. 
\end{claim}

\begin{proof} Consider a $k$-dominating independent set $D$ in $G$, and let $D':=V(G)\setminus D$. Delete the edges of $G\mid_{D'}$. In the resulting graph, the degree $d(v)$ of every $v\in D'$ is at least $k$, while $N(D')=D$ is an independent set. The claim thus follows.
\end{proof}

For the family of  trees, we saw in Section 2. that the number of $1$-DISes can be exponential in the number of vertices via Theorem \ref{tree}.
 If $k>1$, the situation is completely different from the case $k=1$. Confirming  an extended version of a  conjecture of Pawe{\l} Bednarz \cite{pawel} on $k$-DISes of trees, we can formulate the following

\begin{theorem}\label{fa} Let $k>1$.
If $G$ is a tree (or forest) and there exists a $k$-dominating independent set in $G$, then it is unique. That is, $\mi_k(n, \mathcal{T})=1$.
\end{theorem}

\begin{proof} Assume to the contrary that there exists  a forest $T$ with (at least) two different $k$-dominating sets $D_1$ and $D_2$, moreover $T$ is a minimal counterexample regarding the number of vertices and edges. We introduce the notions $L_T$ for the set of leaves in $T$, $Q_T:=N(L_T)$ the neighbors of the leaves and $R_T:=V\setminus (L_T \cup Q_T)$ the rest of the vertices. The minimality condition immediately implies 
that $T$ is a tree.  Furthermore $L_T \subseteq D_i$ from the $k$-domination and $Q_T \cap D_i = \emptyset$ from the independence of the sets $D_i$. % Note that $Q_T$ is an independent set, otherwise deleting the edges between its vertices would provide a counterexample with less edges. Now consider a vertex $v\in Q_T$.  We claim that $N(v) \cap L_T=2$. Indeed, $N(v) \cap L_T<3$ would contradict to the condition that the graph $T$ is minimal, as at least one leaf could be removed and the resulting graph would still have the same number of $2$-dominating sets. On the other hand, in the case $N(v) \cap L_T=1$ we could similarly remove $v$ and its leaf neighbor, and the resulting graph would still have the same number of $2$-dominating sets. 
Consequently, the graph spanned by $R_T$ has at least two different $k$-dominating sets and they can be extended in the same way to $Q_T$ and $L_T$, a contradiction.

Finally, observe that the leaves of a star $S_{k+1}$ form a $k$-dominating independent set in $S_k$.
\end{proof}

An alternative way to see this is a consequence of the following simple Algorithm \ref{alga}, which either finds the unique $k$-dominating set in the tree, or proves that there does not exist any.

\begin{algo}\label{alga} Let $D$ and $D'$ be empty sets in the beginning.

$\bullet$ Put all the isolated vertices to $D$. Cluster the vertices of the forest $T$  to $L_T, Q_T, R_T$ as in the proof of Theorem \ref{fa}.\\ $\bullet$ If $|Q_T|=0$ but $|L_T|>1$, stop with the answer 'no $k$-dominating independent set'.\\ $\bullet$ If $|Q_T|=0$ and $|L_T|=1$,  put $w\in L_T$ to $D$ and stop with the answer '$D$ is the $k$-dominating independent set'.\\ $\bullet$ If $|Q_T|=|L_T|=0$, stop with the answer '$D$ is the $k$-dominating independent set'.\\ $\bullet$ Else choose a vertex $q$ from $Q_T$ which has at most $1$ neighbor from $R_T \cup Q_T$. Note that such a vertex clearly exists since the graph $G \setminus L_T$ is a tree, whose leaf set is a subset of $Q_T$.\\ $\bullet$ $\bullet$ If $|N(q)\cap L_T|\geq k$, then
put $q$ to $D'$ and the vertices of $N(q)\cap L_T$ to $D$, finally delete \mbox{$\{q\}\cup (N(q)\cap L_T)$} from $T$.\\
  $\bullet$ $\bullet$ If $|N(q)\cap L_T|=k-1$ and $|N(q)\cap(R_T \cup Q_T)|=1$, then put $q$ to $D'$ and the vertices of $N(q)\cap L_T$ to $D$, delete \mbox{$\{q\}\cup (N(q)\cap L_T)$} from $T$, and separate the edges which are adjacent to the vertex $N(q)\cap(R_T \cup Q_T)$ in the remaining graph with copies of $N(q)\cap(R_T \cup Q_T)$ as endvertices.\\
 $\bullet$ $\bullet$ Else, stop with the answer 'no $k$-dominating independent set'.\\ Iterate.
\end{algo}

It is easy to see that if a vertex is duplicated, then the copies will be leaves in the remaining graph hence the vertex will be part of $D$ if there exists a suitable $k$-dominating independent set.
It is also clear that $D$ will be an independent set throughout the algorithm. At the same time every vertex in $D'$ will have at least $k$ neighbors from $D$. Indeed, if a vertex $q$ is put into $D'$ when $|N(q)\cap L_T|=k-1$, then it is guaranteed that its last neighbor will be in $D$ as well. Thus $D$ will be a $k$-dominating independent set. 
Finally, the algorithm stops with 'no' answer exactly when there is an evidence for the non-existence.

%-----------------------------------------------------------

\medskip

\section{$k$-DISes --- constructions and lower bounds }

In this section we prove the lower bound of Theorem \ref{fo1} and Theorem \ref{fo2} by showing suitable graphs. 
 Let $G$ be a complete bipartite graph of equal cluster size, or a Tur\'an graph $T_{p^2,p}$ on $p^2$ vertices and $p$ equal partition classes.

\begin{proposition}\label{constr}  $\zeta_k(K_{t,t})=\sqrt[2t]{2}$ if $k\leq t$. 

$\zeta_k(T_{p^2,p})=\sqrt[p^2]{p}$ if $k\leq p$.
\end{proposition}

Putting together  the first statement of Proposition \ref{constr} with $k=t$  and Theorem \ref{alap0} we get the lower bound of Theorem \ref{fo2} : $\zeta_k(K_{k,k})=\sqrt[2k]{2} \leq \lim \zeta_k(n)$. \\
Note that for $k=3$, the Tur\'an graph provides better estimation from  Proposition \ref{constr}:
$\zeta_k(T_{3\cdot 3,3})=\sqrt[9]{3} \leq \lim \zeta_3(n)$. Here $\sqrt[9]{3}\approx 1.13$ while the bipartite graph $K_{3,3}$ would yield only $\sqrt[6]{2}\approx 1.122$.
For $k=4$, $\zeta_4(T_{4\cdot 4,4})= \zeta_4(K_{4,4})$.
\medskip

\noindent Kneser graphs also provide many $k$-DISes:

\begin{proposition}\label{Kneser} Let $G=KN(n,t)$ denote the Kneser graph whose vertices correspond to the $t$-element subsets of a set of $n$ elements, and where two vertices are adjacent if and only if the two corresponding sets are disjoint. Suppose $t< n/2$. Then $G$  contains $n$ $k$-DISes for $k=\binom{n-t-1}{t-1}$.
\end{proposition}

\begin{proof}
Clearly the largest independent set in $G$ is of size $\binom{n-1}{t-1}$ according to the theorem of Erd\H os, Ko and Rado \cite{EKR}, and the corresponding $t$-element subsets are those which contain a fixed element $i\in \{1,2,\ldots, n\}$. Thus the proposition indeed follows since a $t$-element subset which does not contain $i$ are disjoint to exactly $k=\binom{n-t-1}{t-1}$   $t$-subsets which contain $i$, while
 less vertices in a maximal independent set do not provide enough edges to $k$-dominate the rest of the vertices.
\end{proof}

Now we turn our attention to the case $k=2$.

\begin{claim}
 $\mi_2(3)=1$, $\mi_2(4)=2$, $\mi_2(5)=2$, $\mi_2(6)=3$, $\mi_2(7)=3$, $\mi_2(8)=4$, $\mi_2(9)=6$, $\mi_2(16)\geq 24$.
\end{claim}

\begin{proof}
It is easy to check that the number of the $2$-DISes in $P_3$, $K_{2,2}$, $K_{2,2,2}$, $K_{2,2,2,2}$, $K_3\square K_3$ and $K_4\square K_4$  is $1; 2; 3; 4; 6$, and $24$, respectively.  It is also easy to check that joining a new vertex to a graph's every vertex does not increase or decrease the number of the $2$-DISes. Finally, it can be shown by  case analysis that these graphs are extremal indeed.
\end{proof}

Concerning $2$-DISes, product graphs seem to provide the best  lower bound, at least much better than those provided by Proposition \ref{constr}.

\begin{construction} \label{pelda1}
Let $n$ be large enough, and let

$$G_n=  \alpha K_3\square K_3 +\beta K_4\square K_4, \mbox{ \ with \ } \alpha, \beta \in \mathbb{N}, \beta\leq 8.$$ (Observe that $\alpha$ and $\beta$ is uniquely determined.)  

\end{construction}

In view of Observation \ref{obsit} this implies
 
\begin{proposition}\label{order}

$${\mi_2(n)}=  \Omega({6}^{n/9})  \ \mbox{ and hence } \sqrt[9]{6}\leq \lim\zeta_2(n).  $$
\end{proposition}

We conjecture that in fact $\sqrt[9]{6}= \lim\zeta_2(n)$ holds, moreover, the graphs listed in Construction \ref{pelda1} are extremal graphs, that is, if $n$ is large enough then 
$\mi_2(n)=\mi_2(G)$ holds for an $n$-vertex graph only if $G$ is a  graph from Construction \ref{pelda1}.

Concerning $k=3$, we conjecture that the Tur\'an graph $T_{3\cdot 3,3}$ provides the order of magnitude, as $\sqrt[9]{3}\leq\zeta_3(n)$. Finally, in general we conjecture that if $k$ is large enough, then $\zeta_k(n)=\zeta_k(K_{k,k})=\sqrt[2k]{2}$.

\begin{nota}
Let $G^*$ denote the graph constructed from $G$ by adding a new vertex to its vertex set and joining it to all of the vertices of $G$.
\end{nota} 

\noindent Applying the observation   $\mi_k(v(G), G) = \mi_k(v(G)+1, G^*)$, we have
 
\begin{corollary}
$${\mi_2(n, \mathcal{F}_{con})}=  \Omega({6}^{n/9}).$$
\end{corollary}

\subsection{ Connections to MDS codes}

We begin with some preliminaries about coding theory and MDS codes, for more details we refer to \cite{code}.

\begin{definition} Let $C\subseteq \mathbb{F}_q^n$ be a set of codewords in the vector space $\mathbb{F}_q^n$. Defining the Hamming metric $d(*,*)$ on $\mathbb{F}_q^n$, $d(C)$ --  called the minimal distance --  is $d(C)=\min\{d(c, c') : c, c'\in C, c\neq c'\}$. A code $C\subseteq \mathbb{F}_q^n$ is a $q$-ary $(n,M,d)_q$ code if the dimension of the vector space is $n$, $|C|=M$ and the minimal distance is $d$. 
A code $C$ is linear if it is a subspace of the vector space $\mathbb{F}_q^n$.
\end{definition}

The Singleton bound for  a  $q$-ary $(n,M,d)$ code states that $|C|\leq q^{n-d+1}$. If equality holds, then $C$ is said to be a maximum distance separable code, or simply, an MDS code. 
(Linear) MDS codes are extensively studied, and have strong connections to finite geometries, namely, to the existence of certain arcs in multidimensional projective spaces, see \cite{code}. The problem of determining the number of linear MDS codes in $\mathbb{F}_q^n$ of minimal distance $d$ was essentially posed by Segre, and determined so far only in some special cases \cite{mdses1, mdses2}. We highlight here only 

\begin{proposition}\label{enum} The number of linear $q$-ary $(n,M,2)_q$ MDS codes is $(q-1)^{n-1}$.
\end{proposition}

\noindent Much less is known  about the number of all MDS codes in $\mathbb{F}_q^n$ of minimal distance $d$. 

\bigskip

\noindent Now we return to Construction \ref{pelda1}. One may suggest that similar graph products with multiple terms yield bounds on $\zeta_k(n)$.\\
 Consider $t$ disjoint copies  of  $(K_3)^k$. The set $V( (K_3)^k)$ can be represented by vectors over $\mathbb{F}_3$ of length $k$, and two of them is adjacent if and only if they differ in exactly $1$ coordinate. Hence a subset $D$ of $V( (K_3)^k)$ is a  $k$-DIS if and only if every fixed $k-1$ coordinate determines exactly one element of $D$. In other words, $D$ is a set of $3^{k-1}$ vectors from $\mathbb{F}_3^k$, with minimal Hamming distance $2$. Consequently, $D$ is a  MDS code, and the number of $k$-DISes in $ (K_3)^k$ is the number of (not necessarily linear) $(k, 3^{k-1}, 2)_3$ MDS codes. 

\begin{theorem}\label{mds} The number of $(k, 3^{k-1}, 2)_3$ MDS codes is $3\cdot2^{k-1}$.
\end{theorem}

\begin{proof} We prove by induction on $k$.  For $k=1$, the statement clearly holds.\\
First observe that  Proposition \ref{enum} provides $3\cdot2^{k-1}$  general $(k, 3^{k-1}, 2)_3$ MDS codes.  Indeed, any linear MDS code $C$ contains the all-zero vector, and their  translations $C+(0,\cdots, 0,1)$ and  $C-(0,\cdots, 0,1)$
yields suitable new codes. \\
Hence it is enough to  prove that the number of $(k+1, 3^{k}, 2)_3$ MDS codes is at most twice the number of  $(k, 3^{k-1}, 2)_3$ MDS codes if $k\geq 1$. Observe that if one prescribes the value of arbitrary $k$ coordinates in a $(k+1, q^{k}, 2)_q$ MDS code, then exactly one codeword will fulfill the condition.
 Consider a $(k+1, 3^{k}, 2)_3$ MDS code. Observe that the set of codewords having zero as first coordinate are corresponding to a $(k, 3^{k-1}, 2)_3$ MDS code. Indeed, the minimal distance does not change while deleting the first coordinate yields a set of $3^{k-1}$ codewords of length $k$. Finally we prove that such a $(k, 3^{k-1}, 2)_3$ MDS code could be obtained from at most two  $(k+1, 3^{k}, 2)_3$ MDS codes. To this end, delete the first coordinate of the codewords, and omit those codewords which had $0$ on the first coordinate. Thus we get $2\cdot 3^{k-1}$ vectors in $\mathbb{F}_3^k$. Assign a graph $G$ to this vector set by connecting every pair of vectors which are at Hamming distance $1$. The number of proper two-colorings of this graph by colors '1' and '2' is equivalent to the number of extensions of this vector set by an appropriate first coordinate to get a $(k+1, 3^{k}, 2)_3$ MDS code together with the omitted codewords. Notice that the number of proper two-colorings is at most two for any connected graph. Thus Lemma \ref{finish}
finishes the proof.
\end{proof}

\begin{lemma}\label{finish} $G$, the graph assigned to the codewords of nonzero first coordinate, is  connected.
\end{lemma}

\begin{proof}
We prove by contradiction.  Assume that $ (v_1, v_2, \ldots v_{k+1})$ and $(w_1, w_2, \ldots, w_{k+1})$ are codewords, $v_1\neq 0 \neq w_1$, furthermore $\textbf{v}=(v_2, \ldots v_{k+1})$ and $\textbf{w}=(w_2, \ldots, w_{k+1})$ are in different component of $G$ and their Hamming distance is minimal w.r.t. pairs of codewords  taken from different components of $G$. 
Note that $\textbf{v}$ and $\textbf{w}$ must differ in at least two coordinates according to our assumption, hence $k\geq 2$.  W.l.o.g. we may assume that $v_2\neq w_2$. Let us define $z_2$  by $\{v_2, w_2, z_2\}=\{1,2, 0\}$. Since $\textbf{v}$ and $\textbf{w}$ were at the smallest Hamming distance, 
$(w_2, v_3, v_4, \ldots v_{k+1})$ or $(v_2, w_3, w_4, \ldots, w_{k+1})$ cannot be vertices of $G$ since it would yield a smaller Hamming distance. But any $k$ prescribed coordinates can be extended to get a codeword in an $(k+1, 3^{k}, 2)_3$ MDS code, thus $(0, w_2, v_3, v_4, \ldots v_{k+1}), (0, v_2, w_3, w_4, \ldots, w_{k+1}) \in C$.  Hence $(0, z_2, v_3, v_4, \ldots v_{k+1})$ and $(0, z_2, w_3, w_4, \ldots, w_{k+1})$ do not belong to $C$, which implies that $(z_2, v_3, v_4, \ldots v_{k+1})$ and $(z_2, w_3, w_4, \ldots, w_{k+1})$ are in the vertex set of $G$. Observe that they have more common coordinates than $\textbf{v}$ and $\textbf{w}$ had while they still belong to different components, which is a contradiction.
\end{proof}

\begin{remark} The proof implies that the graph assigned to the codewords of nonzero first coordinate is  bipartite as well, and all $(k, 3^{k-1}, 2)_3$ MDS codes are the translates of linear $(k, 3^{k-1}, 2)_3$ MDS codes.
\end{remark}

\begin{corollary} $\zeta_k{((K_3)^k)}=   \sqrt[3^k]{3\cdot2^{k-1}}$. If $k>2$, this is less then $\zeta_k{(K_{k,k})}=   \sqrt[2k]{2}$.
\end{corollary}

\bigskip

\subsection{ Connections to finite geometries}

In this subsection we study constructions coming from finite geometries. The first reason to do this is the fact that  many extremal structures are  provided by geometric constructions in general (see \cite{FS}). In our case  they provide a  graph family with large number of $k$-DISes.
The second reason is that these families have remarkable connections to many interesting subfields of projective geometry, including $m$-fold blocking sets, arcs and tangent-free sets. 

\begin{definition}
Let $PG(2,q)$ denote a finite projective plane over $\mathbb{F}_q$, with point set $\mathcal{P}$ and line set $\mathcal{L}$.  Let $G(\mathcal{P, L})$ be the (bipartite) point-line incidence graph of the geometry.  Note that $G(\mathcal{P, L})$ is a $q+1$ regular graph on $N=2(q^2+q+1)$ vertices.
\end{definition}

\begin{definition}
An $m$-fold blocking set $B$ in a projective plane is a set of points such
that each line contains at least $m$ points of $B$ and some line contains exactly $m$
points of $B$.
\end{definition}

\begin{definition}
In a finite projective plane of order $q$,  a $\{K;t\}$-arc is a nonempty proper subset  $\mathcal{K}$ of $K$ points of the plane such that every line intersects $\mathcal{K}$ in at most $t$ points and there exists a set of $t$ collinear points in $\mathcal{K}$. A $\{K, 2\}$-arc is simply
called a $K$-arc. Note that $\{K, t\}$-arcs and multiple blocking sets are 
complements of each other in a projective plane, that is, the complement of a  $\{K, t\}$-arc  is a  $(q+1-t)$-fold  blocking set. A  $\{K;t\}$-arc is called maximal, if $K=(q+1)t-q$, that is, in the case when the size attains the possible maximum \cite{cossu}. 
\end{definition}

It is well known that every line intersects $\mathcal{K}$ in either $0$ or $t$ points  in a maximal $\{K;t\}$-arc $\mathcal{K}$ \cite{cossu}.  Denniston showed \cite{Denniston} that maximal $\{K;t\}$-arcs exist in  projective planes  $PG(2,q)$ of even order for all divisors $t$ of $q$. On the other hand, Ball, Blokhuis and Mazzocca proved that no maximal $\{K;t\}$-arcs exists in projective planes of odd order \cite{BBM}.

\begin{construction} \label{hyper}
Consider a hyperoval $\mathcal{H}$ in $PG(2,q)$, $q>2$ even, that is, a maximal arc of $q+2$ points. Let the set $D\subseteq V(G)$ consist of the lines skew to $\mathcal{H}$ and the  points of $\mathcal{H}$.  
\end{construction}

\begin{claim} Construction \ref{hyper} provides a $2$-DIS for any hyperoval of the projective geometry.
\end{claim}

\begin{proof}
Any line intersects a hyperoval in $0$ or $2$ points, thus the secants  of the hyperoval are dominated by exactly $2$ vertices  of  $D\cap\mathcal{P}$. The points of  $\mathcal{P}\setminus \mathcal{H}$ are also dominated by at least $2$ vertices of $D\cap\mathcal{L}$ since exactly $q+1-\frac{q+2}{2}$ skew lines are going through 
any external point of $\mathcal{H}$. Finally, it is clear that the set of skew lines and the vertices of $\mathcal{H}$ form an independent set in $G(\mathcal{P, L})$.
\end{proof}

%\begin{claim} The number of hyperovals in $PG(2,q)$ is at least ... polynomial in $q$. Note that the number of regular hyperovals is $\approx q^4$ in $PG(2,q)$.
%\end{claim}

There exist other suitable $2$-dominating (or $k$-dominating) independent sets in $G(\mathcal{P, L})$.\\ Let us take a point set $\mathcal{Q}\subseteq \mathcal{P}$ and the lines skew  to $\mathcal{Q}$ from $\mathcal{L}$. This provides a $k$-dominating independent set of $G(\mathcal{P, L})$ if and only if the following conditions hold:

\begin{itemize}
\item[(1)] Any line intersects $\mathcal{Q}$ in $0$ or at least  $k$ points,
\item[(2)] There exist at least $k$ skew lines to $\mathcal{Q}$ through any point in $\mathcal{P}\setminus\mathcal{Q}$.
\end{itemize}

\begin{corollary}\label{kicsi}
If $\mathcal{Q}$ is  a set without tangents on at most $2q-2$ points, the conditions above  hold for $k=2$. 
\end{corollary}

\noindent Indeed, $(1)$ holds by definition, while $(2)$ is easy to check since if $l$ lines intersect $\mathcal{Q}$ through a given point in $\mathcal{P}\setminus\mathcal{Q}$, then $|\mathcal{Q}|\geq 2l$ must hold.
\medskip

Beside hyperovals of planes of even order, various families of sets are known which fulfill the conditions (1) and (2).
First, consider the generalization of Construction \ref{hyper}.

\begin{construction} \label{k-arc}
 Consider a maximal $\{K;t\}$-arc $\mathcal{K}$ in $PG(2,q)$, $q$ even. Let $G(\mathcal{P, L})$ be the point-line incidence graph of the geometry, and let the set $D\subseteq V(G)$ consists of the lines skew to $\mathcal{K}$ and the  points of $\mathcal{K}$.  
\end{construction}

\begin{claim} Construction \ref{k-arc} provides a $t$-dominating independent set for any maximal $\{K;t\}$-arc of the projective geometry if $t\leq \sqrt{q}$. 
\end{claim}

\noindent Indeed, $(1)$ holds by definition. Concerning $(2)$, at most $q+1-t$ lines can intersect $\mathcal{K}$ through a given point in $\mathcal{P}\setminus\mathcal{K}$, thus $(q+1-t)t\geq |\mathcal{K}|=t(q+1)-q \Leftrightarrow q\geq t^2$.
\smallskip

The so-called $(q+t, t)$-arcs of type $(0,2,t)$ were investigated by Korchm\'aros, Mazzocca, G\'acs and Weiner \cite{korchmaros, gacs}. These are pointsets of $q+t$ points in $PG(2,q)$ such that every line meets them in either $0, 2$ or $t$ points, $2<t<q$. It is easy to see that a necessary condition for their existence is that $t$ divides $q$ and $q$ is even. In \cite{korchmaros} the authors construct an infinite series of examples whenever the field $GF(q/t)$ is a subfield of $GF(q)$. G\'acs and Weiner \cite{gacs} added further geometric and algebraic constructions, moreover, applying a projecting method to maximal $\{2^s(q+1)-q, 2^s\}$-arcs, they presented $(q^{h-1}(2^s(q+1)-q), q^{h-1}2^s)$-arcs of type $(0,2^s,2^sq^{h-1})$  with $h\in \mathbb{Z}^+$. Observe that these sets are examples for $k$-DISes with $k>2$ as well.

So far, we have seen tangent-free sets only if $q$ is even. For  any odd prime power $q>5$, Blokhuis, Seres and Wilbrink presented
 a suitable set of $2q-2$ points arises from the symmetric difference of two conics \cite{BSW}, which provides a $2$-DIS via Corollary \ref{kicsi}. For $q$ prime, no example is known having fewer vertices. 
  If $q=p^h$, $h>1$, Lavrauw, Storme and Van de Voorde constructed a set without tangents of size $q+(q-p)/(p-1)<2q-2$ \cite{LSV}. Up to now, this is the smallest known tangent-free pointset in the $q$ odd case.
   The main idea was to apply the following result. Consider a set $\mathcal{S}$ of $q$ affine points in $PG(2, q)$, $p > 2$, and let $D$ be the set of determined directions of $\mathcal{S}$, lying on the ideal line . If $|D| < (q + 3)/2$,
then  $\mathcal{S}$  together with the complement of $D$ w.r.t. the ideal line is a set without tangents. This was observed and applied by Blokhuis, Brouwer and Sz\H onyi \cite{BBS}, showing a set without tangent of size $2q-q/p$.

  %Van de Vorde \cite{VV}

%----------------------------------------------------------------------------------

\section{Proof of the upper bounds of Theorem \ref{fo1} and \ref{fo2}}

In Section 4 we proved a lower bound on $\mi_2(n)$ in Proposition \ref{order} which provides  $\Theta(1,22^n)<~\mi_2(n)$. This section is devoted to the results on upper bounds. Following the idea of F\"uredi \cite{furedi}, the approach is inductive. We begin with a general upper bound which highlights the key concept.

\begin{proposition}\label{upper1}
Let $\alpha_k:=\max_{d\in \mathbb{Z}^+} \{ \sqrt[d+1]{\frac{k+d}{k}}\}$. Then
$\mi_k(n)= O(\alpha_k^n)$.
\end{proposition}

\begin{proof}
Let $\delta$ denote the minimal degree in a graph $G$, and let $v$ be a vertex of minimal degree in $G$. Any $k$-dominating independent set of $G$ contains either $v$ and none of $N(v)$, or at least $k$ vertices of $N(v)$. The number of $k$-DISes containing $v$ is evidently at most $\mi_k(n-\delta-1)$, while the number of $k$-DISes not containing $v$ is at most $\frac{\delta}{k}\mi_k(n-\delta-1)$. Indeed, any  $w\in N(v)$ appears in at most $\mi_k(n-\delta-1)$ $k$-DISes, and the $k$-dominating property concerning the vertex $v$ implies that we counted any such $k$-dominating independent set at least $k$ times. Hence $\mi_k(n)\leq (1+\frac{\delta}{k})\mi_k(n-\delta-1)$, and the statement follows.
\end{proof}

\begin{remark} Comparing this result with Theorem \ref{alap},  Proposition \ref{upper1} determined   the  right order of magnitude in the case $k=1$.
\end{remark}

\begin{corollary} $\mi_2(n)< \sqrt[3]{2}^n \ \ \mbox{where} \ \   \sqrt[3]{2}\approx 1,26$.
\end{corollary}

In order to prove the upper bound of Theorem \ref{fo1}, we refine the above result. The main idea is to improve the bounds if the minimal degree is less then $4$.

\begin{theorem}\label{upperb2}
 $\mi_2(n)<\sqrt[5]{3}^n \ \ \mbox{where} \ \   \sqrt[5]{3}\approx 1,2457.$
\end{theorem}

\begin{proof}

Define $\tau:=\sqrt[5]{3}$. We prove by induction. Note that $\mi_2(0)\leq \tau^0$ and $\mi_2(1)\leq \tau^1$ trivially holds and assume that $\mi_2(i)\leq \tau^i$ holds for $i=0,\ldots , n-1$.  
Notice that the deletion of possible isolated vertices does not affect the number of $2$-DISes.

Assume first that $\delta=1$ in $G$. Consequently,  $\mi_2(n, G)\leq \mi_2(n-2)\leq \tau^{n-2}\leq \tau^{n}$ as vertices of degree $1$ must be in the $2$-dominating set in contrast with their neighbors.

Next, suppose that $d(v)=\delta=2$. This implies 
\begin{equation}\label{ketto}
\mi_2(n, G)\leq \mi_2(n-3)+\mi_2(n-4)\end{equation}
since the $2$-DISes are either formed by $v$ and a $2$-DIS in $G\setminus N[v]$ or formed by $w_1, w_2 \in N(v)$ and a $2$-DIS in $G\setminus ~( N[w_1]\cup~N[w_2])$. Let $\tau_1$ be the unique positive  root of $P(x)=x^4-x-1$.  ($\tau_1\approx 1,22$.) Then  inequality \eqref{ketto} implies that $\mi_2(n, G)\leq \mi_2(n-3)+\mi_2(n-4)\leq \tau^{n-3}+\tau^{n-4}<\tau^n$ as $\tau_1< \tau$.

Let us suppose $d(v)=\delta=3$. If $|N(w_i)\cup~N(w_j)]|\geq 5$ for all pairs of vertices $w_i\neq w_j\in N(v)$ and $N(v)$ is an independent set, then  
\begin{equation}\label{harom}
\mi_2(n, G)\leq \mi_2(n-4)+\mi_2(n-7) + 2\mi_2(n-8).\end{equation}
Indeed, the $2$-DISes are either formed by $v$ and a $2$-DIS in $G\setminus N[v]$, or formed by $w_1, w_2 \in N(v)$ and a $2$-DIS in $G\setminus ~( N[w_1]\cup~N[w_2])$, or formed by $w_1, w_3 \in N(v)$ and a $2$-DIS in $G\setminus ~( N[w_1]\cup~N[w_3]\cup \{w_2\})$, or formed by $w_2, w_3 \in N(v)$ and a $2$-DIS in $G\setminus ~( N[w_2]\cup~N[w_3]\cup \{w_1\})$. Let $\tau_2$ be the unique positive  root of $P(x)=x^8-x^4-x-2$.  ($\tau_2\approx 1,241$.) Then  inequality \eqref{harom} implies that $\mi_2(n, G)\leq~ \mi_2(n-4)+~\mi_2(n-~7)+~2\mi_2(n-~8)\leq \tau^{n-4}+\tau^{n-7}+2\tau^{n-8}<\tau^n$ as $\tau_2< \tau$.\\
What if $|N(w_i)\cup~N(w_j)|\geq 5$ does not hold for some $w_i\neq w_j\in N(v)$? Then every $2$-DIS which does not contain $v$ must contain both $w_i$ and $w_j$. Indeed, one of them must be in the set  $D$ to dominate $v$, but then the other one cannot be $2$-dominated, thus it must be in the $D$ as well. Hence we could bound the number of $2$-DISes by $\mi_2(n-4)+~\mi_2(n-~5)$, and the inequality $\mi_2(n, G)<\tau^n$ follows easily. \\
Finally, we have to handle the case when $|N(w_i)\cup~N(w_j)|\geq 5$ holds for every $w_i, w_j\in N(v)$ but $N(v)$ induces at least one edge. W.l.o.g, $w_1w_2\in E(G)$ and then we miss the $2$-DISes where $w_1$ and $w_2$ were both part of $D$ in inequality \eqref{harom}, which yields 
\begin{equation}\label{negy}
\mi_2(n, G)\leq \mi_2(n-4)+2\mi_2(n-7)\end{equation}
to hold in this case. Observing that the unique positive root $\tau_3$ of $x^7-x^3-2$ is less then $\tau$, we conclude to
$\mi_2(n, G)<\tau^n$ again.

At last, applying the proof of Proposition \ref{upper1} to $\delta\geq 4$, we get $$\mi_2(n)\leq \left(\frac{2+\delta}{2}\right)\mi_2(n-\delta-1).$$ 

The fact $$\max_{d\in\mathbb{Z}, d\geq 4} \left\{ \sqrt[d+1]{\frac{2+d}{2}}\right\}= \sqrt[5]{\frac{6}{2}}=\tau $$
thus completes the proof.
\end{proof}

\begin{proof}[Proof of Theorem \ref{fo2}, upper bound]
Finally, to obtain the upper bound in Theorem \ref{fo2}, we only have to observe two facts. On the one hand, we  can assume that $\delta\geq k$ holds for the minimal degree of $G$, similarly to the proof of Theorem \ref{upperb2}. Indeed, otherwise we would get $\mi_k(n, G)\leq \mi_k(n-\delta)$. On the other hand, easy computation shows that $\sqrt[d+1]{\frac{k+d}{k}}$ is a  monotone decreasing function of $d$ from $d=k$, if $k$ is fixed. Thus  $\mi_k(n, G)\leq 2\cdot \mi_k(n-k-1)$, and the upper bound follows.
\end{proof}

\section{Concluding remarks and open problems}

In this final chapter we gather some problems and conjectures related to the discussed results.

\begin{problem} \label{mdscode}
Determine or bound the number of all MDS codes, especially the number of $q$-ary MDS codes of type $(n,M,2)_q$

%\end{itemize}
\end{problem} 

\begin{remark}
The result is related to the number of $q-1$-coloring of certain Hamming-graphs in view of the proof of Theorem \ref{mds}. Note that this problem is widely open even if we consider linear MDS codes, and on the other hand $q$ is not required to be a prime power.
\end{remark}

\begin{problem} \label{vegessik}
Determine the number of maximal independent sets of the incidence graph $G(\mathcal{P, L})$ of the projective geometry $PG(2,q)$ in terms of the number of vertices.
\end{problem} 

\begin{conjecture} For every $k$, there exists a graph $G$ for which $\zeta_k(G)=\lim \zeta_k(n)$ holds.
\end{conjecture}

\begin{conjecture}($\sqrt[9]{6}$-conjecture) The maximal number of $2$-DISes in $n$-vertex graphs is $\Theta(\sqrt[9]{6}^n)$. That is, $\zeta_2(n)=\zeta_2(K_3\square K_3)$. Moreover, Construction \ref{pelda1} provides the extremal graphs for the function $\mi_2(n)$ if $n$ is large enough.
\end{conjecture} 

\begin{conjecture} The maximal number of $k$-DISes in $n$-vertex graphs is attained for the disjoint union of $K_{k,k}$ graphs for $k>3$ if $2k|n$.
\end{conjecture}

\begin{problem} Describe large graph families $\mathcal{F}$ for which 
\begin{itemize}
\item $\mi_k(n, \mathcal{F})\leq 1$,
\item $\mi_k(n, \mathcal{F})$ is bounded by a  polynomial of $n$,
\item $\lim\zeta_k(n, \mathcal{F})=1$.
\end{itemize}
\end{problem} 
This problem is motivated by the results of Farber, Hujter and Tuza \cite{HT2}.

\begin{conjecture} $\mi_k(n, \mathcal{F})$ is not bounded by a polynomial of $n$ for the graph family of incidence graphs of projective planes.
\end{conjecture}

\bigskip
\noindent
{\bf Acknowledgments}

I would like to thank Zolt\'an F\"uredi, Tam\'as H\'eger, Mikl\'os Simonovits, Tam\'as Sz\H{o}nyi and Zsolt Tuza for helpful discussions on the topics of this paper.
\bigskip

%Tam\'as H\'eger, Tam\'as Sz\H onyi, ...

%The research of the first author was supported by the Australian Research
%Council, by ERC Advanced Research Grant No. 267165, and by Hungarian
%National Scientific Research Funds (OTKA) Grant No. 100291.

%% References
%%

\end{document}